\documentclass[12pt,a4paper]{amsart}
\usepackage{amssymb,latexsym,amsmath,amscd,graphicx,url,color}
\let\oldlabel=\label
\def\prellabel{\marginparsep=1em\marginparwidth=44pt
    \def\label##1{\oldlabel{##1}\ifmmode\else\ifinner\else
         \marginpar{{\footnotesize\ \\ \tt
                    ##1}}\fi\fi}}

\theoremstyle{plain}
\newtheorem{thm}{Theorem}[section]
\newtheorem{prop}[thm]{Proposition}
\newtheorem{cor}[thm]{Corollary}
\newtheorem{lemma}[thm]{Lemma}
\newtheorem{conjecture}[thm]{Conjecture}

\theoremstyle{definition}
\newtheorem{defn}[thm]{Definition}

\newtheorem{ex}[thm]{Example}

\newcommand{\AAA}{{\mathbb A}}

\newcommand{\M}{{\mathcal M}}

\newcommand{\NN}{{\mathbb N}}
\newcommand{\PP}{{\mathbb P}}
\newcommand{\QQ}{{\mathbb Q}}
\newcommand{\RR}{{\mathbb R}}

\newcommand{\ZZ}{{\mathbb Z}}
\newcommand{\one}{{\bf 1}}
\newcommand{\mm}{\mathfrak m}

\DeclareMathOperator{\CS}{CS}
\DeclareMathOperator{\HS}{HS}

\DeclareMathOperator{\GCD}{GCD}
\DeclareMathOperator{\LCM}{LCM}
\DeclareMathOperator{\gin}{gin}
\DeclareMathOperator{\inid}{in}

\DeclareMathOperator{\gdeg}{gdeg}

\DeclareMathOperator{\Ker}{Ker}

\DeclareMathOperator{\reg}{reg}

\DeclareMathOperator{\rank}{rank}
\DeclareMathOperator{\MDeg}{Deg}
\DeclareMathOperator{\GDeg}{GDeg}
\DeclareMathOperator{\length}{length}

\title{Cartwright-Sturmfels ideals associated to graphs and linear spaces}

\author{A. Conca}
\address{Dipartimento di Matematica, 
Universit\`a di Genova, Via Dodecaneso 35, 
I-16146 Genova, Italy}
\email{conca@dima.unige.it}

\author{E. De Negri}
\address{Dipartimento di Matematica, 
Universit\`a di Genova, Via Dodecaneso 35, 
I-16146 Genova, Italy}
\email{denegri@dima.unige.it}

\author{E. Gorla}
\address{Institut de Math\'ematiques, Universit\'e de Neuch\^atel, Rue Emile-Argand 11, CH-2000
  Neuch\^atel, Switzerland}  
\email{elisa.gorla@unine.ch}

\thanks{The first two authors were partially supported by INdAM-GNSAGA. The third author was partially supported by the Swiss National Science
  Foundation under grant no. 200021\_150207.}


\begin{document}

\begin{abstract}
Inspired by work of Cartwright and Sturmfels, in~\cite{CDG2} we introduced two classes of multigraded ideals named after them.  These ideals  are defined in terms of properties of their multigraded generic initial ideals. The goal of this paper is showing that three families of ideals that have recently attracted the attention of researchers are Cartwright-Sturmfels ideals. More specifically, we prove that  binomial edge ideals, multigraded homogenizations of linear spaces, and multiview ideals are Cartwright-Sturmfels ideals, hence recovering and extending recent  results of Herzog, Hibi, Hreinsdottir, Kahle, and Rauh \cite{HHHKR}, Ohtani \cite{O}, Ardila and Boocher \cite{AB},  Aholt, Sturmfels, and Thomas \cite{AST}, and Binglin Li \cite{Binglin}. We also propose a conjecture on the rigidity of local cohomology modules of Cartwright-Sturmfels ideals, that was inspired by a theorem of Brion. We provide some evidence for the conjecture by proving it in the monomial case. 
\end{abstract}

\maketitle

\section*{Introduction}

Inspired by the work of Cartwright and Sturmfels \cite{CS}, in~\cite{CDG2} we introduced two families of multigraded ideals, namely Cartwright-Sturmfels ideals and Cartwright-Sturmfels$^*$ ideals. Both families are characterized  by means of  multigraded generic initial  ideals. Cartwright-Sturmfels (CS for short) ideals are the $\ZZ^n$-graded ideals whose multigraded generic initial  ideal is radical, while Cartwright-Sturmfels$^*$ (CS$^*$)  ideals are the $\ZZ^n$-graded ideals whose multigraded generic initial ideal has a system of generators which involves at most one variable of degree $e_i\in \ZZ^n$ for every $i=1,\dots,n$.

Ideals with a minimal system of generators that  is also a universal Gr\"obner basis are called robust.  Being robust is a very strong property and robust ideals have attracted a lot of interest, especially in recent years, see e.g.~\cite{BZ, BR, CDG1,CDG2,CDG3, K, S, Su, SZ}. 
Cartwright-Sturmfels$^*$ ideals are ``very" robust in the sense that every multigraded minimal system of generators of a CS$^*$ ideal is a universal Gr\"obner basis. Furthermore a Cartwright-Sturmfels$^*$ ideal has the same graded Betti numbers as its initial ideals. 
It turns out that  CS ideals and their initial ideals are radical and very often universal Gr\"obner bases of CS ideals can be kept under control as well. 
%These  results are proved in ~\cite{CDG1,CDG2}, where we recover and extend results from~\cite{SZ,BZ,Bo} related to universal Gr\"obner bases of ideals of maximal minors.

In \cite{CDG1} and  \cite{CDG2} we presented large families of determinantal ideals that are CS or CS$^*$, and discussed how these results generalize classical results of Bernstein, Sturmfels, and Zelevinsky \cite{SZ, BZ} on universal Gr\"obner bases for generic maximal minors. In \cite{CDG3}  we gave a combinatorial description of their multigraded generic initial ideals. 

The goal of this paper is showing that the following three families of ideals, that have recently attracted the attention of  several researchers, are CS and, in some cases, also CS$^*$.  

\begin{enumerate}
\item  {\bf Binomial edge ideals.}  Binomial edge ideals have been introduced and studied by Herzog, Hibi, Hreinsdottir, Kahle, and Rauh~\cite{HHHKR} and independently by Ohtani \cite{O}, who proved that they are always radical. Matsuda and Murai  \cite{MM} proved that the regularity of the binomial edge ideal associated to a graph $G$ is bounded by the number of vertices of the graph and conjectured that equality holds only when the graph is a line. Other interesting results concerning  binomial edge ideals can be found, for example, in \cite{EZ, EHHQ,KM}.  We prove that binomial edge ideals  are CS and describe the associated generic initial ideal. As an immediate corollary we obtain the aforementioned results of  \cite{HHHKR, MM, O}. \medskip  

\item  {\bf Closure of linear spaces in products of projective spaces.}  Let $V$ be a vector space of linear forms in $K[x_1,\dots, x_n]$  and let $L\subset \AAA_K^n$ be the zero locus of $V$, i.e. a linear space containing the origin. In \cite{AB} Ardila and Boocher studied the ideal $I(\tilde{L})$ defining the closure $\tilde{L}$ of $L$  in $(\PP^1)^n$. They established several interesting structural properties of  $I(\tilde{L})$. 
We prove that $I(\tilde{L})$ is both a CS and a CS$^*$ ideal. As a consequence we recover some of the results of \cite{AB}. More generally we prove that the ideal $I_a(\tilde{L})$  defining the closure  of $L$ in  the  product $\PP^{a_1}\times \dots \times \PP^{a_u}$ where $a=(a_1,\dots, a_u)$ and $\sum a_i=n$ is CS (but not CS$^*$ in general). We describe $I_a(\tilde{L})$ first as the saturation with respect to the set of homogenizing variables of a single determinantal ideal and then as a sum of several determinantal ideals. Furthermore we give a combinatorial characterization of  its multidegree, or equivalently, a description of the multigraded generic initial ideal of  $I_a(\tilde{L})$. Finally, we observe that $I_a(\tilde{L})$ defines a Cohen-Macaulay normal domain. This follows from Brion's Theorem, as we discuss in Section \ref{S1}. \medskip   

\item  {\bf Multiview ideals.} A collection of matrices $A=\{A_i\}_{i=1,\dots,m}$ with scalar entries, with $A_i$ of size $d_i\times n$ and $\rank A_i=d_i$, induces a rational map
$$\phi_A: \PP^{n-1}  \dashrightarrow   \prod \PP^{d_i-1}$$ sending $x\in \PP^{n-1}$ to $(A_ix)_{i=1,\dots, m}$. 
The ideal $J_A$ of the closure of the image of $\phi_A$ is called multiview ideal. As explained in \cite{AST}, the ideal  $J_A$  plays an important role in various aspects of geometrical computer vision. In \cite{AST} it is proved that $J_A$ is a CS ideal when $n=4$, $d_i=3$ for all $i$, and assuming that the $A_i$'s are generic. In \cite{Binglin} Binglin Li proved results that, suitably interpreted,  imply that   $J_A$ is a CS ideal in all cases. We show that the same conclusion can be obtained as a simple corollary of our results on CS ideals in two ways:  via elimination from the fact that the ideal of $2$-minors of a multigraded matrix is CS  and, again,  via elimination from the multigraded closure of  a linear space  as discussed in (2). 
\end{enumerate} 

Notice that the results in (2) and (3) also answer some of the questions posed by Ardila and Boocher \cite[pg. 234]{AB}. 

%We proved that the class of CS ideal is closed under several  operations, including arbitrary multigraded linear sections, and we also gave a bound on their Castelnuovo-Mumford regularity, i.e. a $\ZZ^n$-graded CS ideal has regularity $\leq n$. Furthermore we proved that ideals of $2$-minors and the ideal of maximal minors of multigraded matrices are CS. These results can be found in~\cite{CDG1,CDG2,CDG3}. 
A result of Brion \cite{Brion} asserts that a prime ideal with multiplicity-free multidegree defines a Cohen-Macaulay normal domain and is a CS ideal. As a corollary of Brion's Theorem, in Section \ref{S1} we show that the minimal primes of a CS ideal are CS as well.
As a further consequence we get that, if $P$ is a prime CS ideal, then every ideal with the same multigraded Hilbert function of $P$ is Cohen Macaulay. We propose two conjectures (Conjecture \ref{conjB} and Conjecture \ref{conjH}) concerning extremal Betti numbers and the Hilbert function of local cohomology modules of CS ideals, that are inspired by Brion's Theorem and confirmed by extensive computations. We prove that both conjectures hold for CS monomial ideals. 
 
Our results have been suggested and confirmed by extensive computations performed using CoCoA \cite{CoCoA} and Macaulay2 \cite{M2}. 
We thank Michel Brion, Marc Chardin, Kohji Yanagawa, and Matteo Varbaro for useful discussions and suggestions.

\section{Multidegrees and CS ideals}\label{S1}

Let $K$ be a field and $S=K[ x_{ij} : i=1,\dots, n, 1\leq j\leq d_i]$ with the standard  $\ZZ^n$-graded structure induced by $\deg(x_{ij})=e_i\in\ZZ^n$.
We assume that $d_1,\dots, d_n>0$ and set $x_i=x_{i,1}$.  Let $T=K[x_{1},x_{2},\dots, x_{n}]\subset S$ with the induced standard  $\ZZ^n$-graded structure. For $m\in \NN$ we set $[m]=\{1,\dots,m\}$ and $[m]_0=\{0,\dots,m\}$. 
A prime ideal $P$ of $S$ is called relevant if $P\not\supset S_{(1,\dots,1)}$ and irrelevant otherwise. 

For a $\ZZ^n$-graded $S$-module $M$, denote by $M_a$ the homogeneous component of $M$ of degree $a\in \ZZ^n$.  

\begin{defn}
Let $M$ be a finitely generated, $\ZZ^n$-graded $S$-module and set  
$c=\dim S-\dim M$.  The {\em $\ZZ^n$-graded Hilbert series} of $M$ is
$$\HS(M,z)=\sum_{a\in \ZZ^n} (\dim_K M_a)  z^a\in\QQ[[z_1,\dots,z_n]][z_1^{-1},\dots, z_n^{-1}].$$ 
\end{defn} 
Set 
$$K_M(z)= \HS(M,z)\prod_{i=1}^n (1-z_i)^{d_i}.$$
It turns out that $K_M(z) \in \ZZ[z_1^{\pm 1} ,\dots,z_n^{\pm 1}]$. Then set 
$$C_M(z)=K_M(1-z_1,\dots,1-z_n).$$ 
The multidegree $\MDeg_M(z)$ of $M$, as defined in \cite[Chapter 8]{MS}, is the homogeneous component of smallest total degree, i.e. of degree $c$,  of $C_M(z)$. It turns out that $\MDeg_M(z)\in \NN[z_1,\dots,z_n]$. Note that by \cite[Claim 8.54]{MS} the multidegree $\MDeg_M(z)$ does not change if one replaces $M$ by a shifted copy of it. Hence it is not restrictive to assume that $M_a=0$  unless $a\in \NN^n$ and, under this assumption, $K_M(z)$ and $C_M(z)$ are actually polynomials. 
In the geometric setting multidegrees are related to Chow classes, see \cite[Notes pg. 172]{MS} for details and references. 

For a Laurent polynomial $G(z)\in \RR[z_1^{\pm 1} ,\dots,z_n^{\pm 1}]$ let us denote by $[G(z)]_{\min}$ the sum of the terms, including their coefficients, that are minimal with respect to division in the support of $G(z)$. For example, if $G(z)=z_1^2+2z_1z_2^2+3z_1^3$, then $[G(z)]_{\min}=z_1^2+2z_1z_2^2$. One easily checks that: 
\begin{lemma}\label{minpart} 
Let $G_1(z),\dots, G_v(z)$ be Laurent polynomials such that $[G_i(z)]_{\min}$ has positive coefficients for every $i$. Then 
$$\left [\sum_{i=1}^v G_i(z)\right ]_{\min}= \left [\sum_{i=1}^v [ G_i(z)]_{\min} \right ]_{\min}.$$
\end{lemma} 

In \cite{CDG1} we defined the G-multidegree of $M$ as follows: 
$$\GDeg_M(z)= [ C_M(z) ]_{\min}.$$ 
Clearly the homogeneous component of degree  $c$ of $\GDeg_M(z)$ is $\MDeg_M(z)$. In Proposition~\ref{flachi1} we will show that  $\GDeg_M(z)=\MDeg_M(z)$ if all the minimal primes of $M$ have codimension $c$. On the other hand, if $M$ has minimal primes of codimension greater than $c$, then $\GDeg_M(z)$ might contain terms of degree higher than $c$. 

By definition $$\MDeg_M(z)=\sum  e_M(a) z^a$$ where  the sum runs over all $a\in \prod_{i=1}^n [d_i]_0$ such that $|a|=c$. It turns out that $e_M(a)\in \NN$. The module $M$ has a multiplicity-free multidegree if $e_M(a)\in \{0,1\}$ for all $a$. Furthermore $M$ has a multiplicity-free G-multidegree if all the non-zero coefficients in $\GDeg_M(z)$ are equal to $1$.  With a slight abuse of terminology we will say that a multigraded ideal $I$ has  multiplicity-free multidegree (or multiplicity-free G-multidegree) if the quotient ring $S/I$ has that property. 

Assuming that  $M_a\neq 0$ for $a\gg 0$ (that is, when all the components of $a$ are sufficiently large), there exists a non-zero polynomial $P_M(z)\in \QQ[z_1,\dots, z_n]$, the multigraded Hilbert  polynomial of $M$, such that $P_M(a)=\dim_K M_a$ for $a\gg 0$. Denote by $D_M(z)$ the homogeneous component of largest degree of $P_M(z)$. If $M$ has  irrelevant minimal primes then there is no clear relation between $D_M(z)$ and $\MDeg_M(z)$. On the other hand if one assumes that $M$ has at least one relevant minimal prime of minimal codimension, then the total degree of $D_M(z)$ is $\dim M-n$ and the coefficients $D_M(z)$ can be deduced from those of $\MDeg_M(z)$.  
Let 
$$D_M(z)=\sum  \frac{f_M(a)}{a!} z^a$$ 
where the sum runs over the $a\in \prod_{i=1}^n [d_i-1]_0$ such that $|a|=\dim M-n$. The numbers $f_M(a)$ are actually non-negative integers and are called mixed multiplicities of $M$. 

The polynomials $D_M(z)$ and $\MDeg_M(z)$ are related as follows: for all the $a\in \prod_{i=1}^n [d_i-1]_0$ such that $|a|=\dim M-n$ one has $f_M(a)=e_M(a')$, where $a'=(d_1-1-a_1, \dots,  d_n-1-a_n)$. 
If  $a\in \NN^n$ is such that  $|a|=c$ and  $a_i=d_i$ for some $i$, then the corresponding coefficient $e_M(a)$  cannot be read off $D_M(z)$.  However, if all the minimal primes of minimal codimension of $M$ are relevant, then such coefficients are actually  zero. Therefore one has: 

\begin{lemma} 
Assume that all the minimal primes of minimal codimension of $M$ are relevant.
Then the polynomials $\MDeg_M(z)$ and $D_M(z)$ are two different encodings of the same numerical data. In particular, this is the case if $M=S/P$ and $P$ is a relevant prime. 
\end{lemma}

If $K$ is algebraically closed and $P$ is a relevant prime ideal then the coefficients $e_{S/P}(a)$ have a geometric interpretation. Let $X$ denote the associated subvariety of $\PP^{d_1-1}\times \dots \times \PP^{d_n-1}$. The coefficient $e_{S/P}(a)$ is the number of points of intersection of $X$ with $L_1\times \cdots \times L_n$ where $L_i$ is a general linear space of $\PP^{d_i-1}$ of dimension $a_i$. 

Given a term order $\tau$ and a $\ZZ^n$-graded homogeneous ideal $I$ of $S$, one can consider its {\em $\ZZ^n$-graded generic initial ideal} $\gin(I)$. As in the $\ZZ$-graded setting, $\ZZ^n$-graded generic initial ideals are Borel fixed. We refer to \cite[Section 1]{CDG2} for more details on multigraded generic initial ideals.  
We just recall that to any $a\in \prod_{i=1}^n [d_i]_0$ one can associate the Borel fixed prime ideal 
$$P_a=( x_{ij} : 1\leq i \leq n \mbox{ and } 1\leq j\leq a_i )$$
 and that any Borel fixed prime  ideal of $S$ is of this form. 

We have:

\begin{lemma}\label{supportgin} 
Let $I$ be a $\ZZ^n$-graded ideal of $S$ and let $\{z^{a_1},\dots, z^{a_s} \}$ be the support of $\GDeg_{S/I}(z)$. Then the minimal primes of the $\ZZ^n$-graded generic initial ideal of $I$ are $\{P_{a_1},\dots,P_{a_s}\}$. 
\end{lemma} 

\begin{proof} Let $J=\gin(I)$. Since the G-multidegree only depends on the Hilbert series we have $\GDeg_{S/I}(z)=\GDeg_{S/J}(z)$. Then the result follows from \cite[Prop. 3.12]{CDG1}. 
\end{proof}

\begin{defn}\label{defCS}
Let $I$ be a $\ZZ^n$-graded ideal of $S$. We say that $I$ is a  Cartwright-Sturmfels (CS) ideal if there exists a radical Borel fixed ideal $J$ of $S$ such that $\HS(I,y)=\HS(J,y)$. 
\end{defn} 

With the notation of Definition~\ref{defCS} it turns out that $J=\gin(I)$, see \cite[Proposition 1.6]{CDG2}. Notice also that by Lemma~\ref{supportgin} one has the following characterization. 

\begin{prop}\label{charaCS}
Let $I$ be a $\ZZ^n$-graded ideal of $S$.  One has that $I$ is CS if and only if $I$ has a multiplicity-free G-multidegree and $\gin(I)$ has no embedded primes. 
\end{prop} 

\begin{defn} \label{defCS+}
We say that $I$ is a Cartwright-Sturmfels$^*$ (CS$^*$) ideal if there exists a $\ZZ^n$-graded ideal $J$ of $S$ extended from $T$ such that $\HS(I,z)=\HS(J,z)$. 
\end{defn}

With the notation of Definition~\ref{defCS+} it turns out that $J=\gin(I)$ and that $I$ and $J$ have the same $\ZZ^n$-graded Betti numbers, see \cite[Proposition 1.9, Corollary 1.10]{CDG2}. 

Notice that a $\ZZ^n$-graded homogeneous ideal of $T$ is just a monomial ideal of $T$. Hence, a $\ZZ^n$-graded ideal of $S$ which is extended from $T$ is an ideal of $S$ generated by monomials in $x_{1},\ldots, x_{n}$.

We observe the following: 

\begin{prop}\label{flachi1}
 Let $M$ be a finitely generated $\ZZ^n$-graded $S$-module and set 
$$F_M(z)=\sum \length(M_P) \MDeg_{S/P}(z)$$
where the sum runs over the minimal primes $P$ of $M$. Then:
\begin{itemize} 
\item[(1)] $\GDeg_M(z)=[ F_M(z)]_{\min}$. 
\item[(2)] If all minimal primes of $M$ have the same codimension, then 
$$F_M(z)=\MDeg_M(z)=\GDeg_M(z).$$
\end{itemize} 
\end{prop}

\begin{proof} 
(1) Consider a multigraded composition series $0=M_0\subset M_1 \dots \subset M_v=M$ such that for every $i$ one has $M_i/M_{i-1}\simeq S/(P_i)(-u_i)$  where $P_i$ is multigraded prime and $u_i\in \ZZ^n$ is a shift. Set $N_i=M_i/M_{i-1}$.  Since $K$-polynomials are additive on short exact sequences, so are $C$-polynomials. Then 
$$C_M(z)=\sum_{i=1}^v C_{N_i}(z).$$ 
Since $[C_{N_i}]_{\min}=\MDeg_{S/P_i}(z)$ independently of the shift $u_i$, and since $\MDeg_{S/P_i}(z)$ has positive coefficients, by Lemma~\ref{minpart} we have 
$$\GDeg_M(z)= [C_M(z)]_{\min}=\left[ \sum_{i=1}^v\MDeg_{S/P_i}(z) \right]_{\min} $$
Observe that if $P_1\subsetneq P_2$ are multigraded prime ideals then 
$$[\MDeg_{S/P_1}(z)+\MDeg_{S/P_2}(z)]_{\min}=\MDeg_{S/P_1}(z)$$
as can be seen by observing that $\gin(P_1)\subset \gin(P_2)$ and using that, by the main result of \cite{KS}, all the minimal primes of 
 $\gin(P_1)$ have minimal codimension. Hence we may remove from $[ \sum_{i=1}^v\MDeg_{S/P_i}(z) ]_{\min}$ those summands  corresponding to primes that are not minimal. Furthermore, by localization, we know  that a given minimal prime of $M$ occurs in  the list $P_1\dots, P_v$ exactly $\length(M_P)$ times. So we obtain the desired formula. 
Assertion  (2) follows from (1) and \cite[Theorem 8.53]{MS}. 
\end{proof}

We may deduce the following important corollary: 

\begin{cor}\label{flachi2}
Let  $I$ be a CS ideal,  then  $$\GDeg_{S/I}(z)=\sum \MDeg_{S/P}(z)$$ where the sum runs over the minimal primes $P$ of $I$.
\end{cor} 

\begin{proof}
Set  $\one=(1,1,\dots, 1)$ and let $F_{S/I}(z)$ be as in Proposition \ref{flachi1}. Since $I$ is radical,  we have that  $F_{S/I}(\one)$ is the geometric degree $\gdeg(S/I)$ in the sense of \cite{STV}. By Proposition~\ref{flachi1} we have $\GDeg_{S/I}(\one)\leq F_{S/I}(\one)$ and equality holds if and only if $\GDeg_{S/I}(z)=F_{S/I}(z)$. Now let $J=\gin(I)$. 
Observe that $\GDeg_{S/I}(z)=\GDeg_{S/J}(z)$ because $\GDeg$ just depends on the Hilbert series. Moreover $\GDeg_{S/J}(\one)=\gdeg(S/J)$ by \cite[Prop.3.12]{CDG1}. Both $I$ and $J$ are radical, hence their geometric degrees coincide with their  arithmetic degrees. Therefore, combining \cite[Prop. 4.1 ]{STV} and \cite[Thm.2.3]{STV} we have $\gdeg(S/I)= \gdeg(S/J)$. Summing up, we have:
$$\gdeg(S/J)=\GDeg_{S/J}(\one)=\GDeg_{S/I}(\one)\leq F_{S/I}(\one)=\gdeg(S/I)= \gdeg(S/J).$$
Hence $\GDeg_{S/I}(\one)=F_{S/I}(\one)$ which implies  $\GDeg_{S/I}(z)=F_{S/I}(z)$. 
\end{proof}

Since every  radical Borel fixed ideal has a multiplicity-free G-multidegree   it follows that the same is true for every CS ideal. Furthermore every CS ideal is generated in degree $\leq \one \in \ZZ^n$. However, there are radical ideals  generated in degree $\leq \one$ and with a multiplicity-free G-multidegree and that are not CS ideals, as the next example shows. 
 
\begin{ex} 
Let $S=\QQ[x_1,x_2,x_3,y_1,y_2,y_3]$ with the $\ZZ^2$-graded structure induced by $\deg(x_i)=e_1$ and $\deg(y_i)=e_2$. The ideal $I=(x_1y_1, x_2y_2, x_3y_2, x_2y_3, x_3y_3)$, generated in degree $\one=(1,1)$, is  the intersection of 
$(x_1, x_2, x_3), (y_1,x_2, x_3), (x_1, y_2, y_3), (y_1, y_2, y_3)$. Hence by Corollary \ref{flachi2} the multidegree of $S/I$ is 
$$\GDeg_{S/I}(z)=\MDeg_{S/I}(z)=z_1^3+z_1^2z_2+z_1z_2^2+z_2^3.$$
On the other hand, its multigraded generic initial ideal is 
$$\gin(I)=(x_1y_1, x_2y_1, x_1y_2, x_2y_2, x_3y_1, x_1x_2y_3, x_1^2y_3)$$
so that $I$ is not CS. 
\end{ex}  

The classes of $\CS$ and $\CS^*$ ideals are in a sense dual to each other. In fact, in~\cite[Theorem~1.14]{CDG2} we showed that if $I$ is a squarefree monomial ideal, then $I$ is CS if and only if its Alexander dual $I^*$ is CS$^*$. Moreover, it follows from the definitions that the families of CS and CS$^*$ ideals are closed under $\ZZ^n$-graded coordinate changes and taking initial ideals. In~\cite{CDG1,CDG2} we showed that if $I$ is CS or CS$^*$, then its $\ZZ^n$-graded generic initial ideal does not depend on the choice of the term order but only on the total order given to the indeterminates with the same degree. We also proved that each of the two classes is closed with respect to a number of natural operations, see~\cite[Proposition~1.7 and Theorem~1.16]{CDG2}. Other interesting properties, including bounds on the projective dimension and Castelnuovo-Mumford regularity were established in~\cite[Proposition~1.9, Proposition~1.12, and Corollary~1.15]{CDG2}.

A beautiful theorem of Brion \cite{Brion} asserts that an irreducible subvariety $X$ of a flag variety is normal and Cohen-Macaulay if it has multiplicity-free multidegree. Moreover Brion showed that such an $X$ admits a flat degeneration to a reduced union of Schubert varieties that is Cohen-Macaulay as well. See also the work of Perrin \cite{P}. Using the terminology that we have introduced and limiting ourselves to subvarieties of a product of projective spaces, Brion's result can be stated as follows: 

\begin{thm}[Brion]\label{Brion}
Assume $K$ is algebraically closed. Let $P$ be a $\ZZ^n$-graded prime ideal in the polynomial ring $S$. Assume that  $S/P$ has a multiplicity-free multidegree.  
Then: 
\begin{enumerate}
\item $S/P$ is normal and Cohen-Macaualy,
\item  $P$ is a CS ideal,  
\item  the multigraded gin of $P$ defines a Cohen-Macaulay ring.
\end{enumerate}
\end{thm}

As a consequence of Brion's Theorem we have: 

\begin{cor}  Assume that $K$ is algebraically closed.
\begin{itemize}
\item[(1)] Let $P$  be a prime CS ideal.  Then every $\ZZ^n$-graded  ideal with the same $\ZZ^n$-graded  Hilbert series of $P$ defines a Cohen-Macaulay ring. 
\item[(2)] Let $I$  be  CS ideal and let $P_1,\dots, P_s$ be its minimal primes.  Then each $P_i$ is CS and 
$$\inid(I)=\cap_{i=1}^s  \inid(P_i)$$
for every term ordering. 
\end{itemize}
\end{cor}

\begin{proof} 
Assertion (1) follows immediately from Theorem~\ref{Brion} since $P$, being CS, has a multiplicity-free multidegree.  To prove (2) one observes that  by Corollary~\ref{flachi2} a minimal prime of a CS ideal has a multiplicity-free multidegree, hence it is a CS ideal by Theorem~\ref{Brion}. 
\end{proof} 

Brion's Theorem suggests that there might be a very strong connection between homological invariants of a CS  ideal  and that of  its generic initial ideal.  Computational experiments suggest the following conjecture: 

\begin{conjecture}\label{conjB}
Let $I$ be a CS ideal and $J$ its $\ZZ^n$-graded generic initial ideal. Then  the extremal (total) Betti numbers of $I$ and $J$ are equal. In particular,  $I$ and $J$ have the same projective dimension and Castelnuovo-Mumford regularity. 
 \end{conjecture}

Even  stronger, we conjecture the following. 

\begin{conjecture}\label{conjH}
Let $I$ be a CS ideal and $J$ its $\ZZ^n$-graded generic initial ideal. Then one has: 
$$\dim_K H^i_{\mm}(S/I)_a=\dim_K H^i_{\mm}(S/J)_a$$
for every $i\in \NN$ and every $a\in \ZZ^n$. 
\end{conjecture}

Here $H^i_{\mm}(S/I)$ denotes the multigraded $i$-th local cohomology module supported on the graded maximal ideal $\mm$ of $S$.  As explained by Chardin in \cite{Char} extremal Betti numbers can be characterized in terms of vanishing of graded components of the local cohomology modules. Therefore Conjecture \ref{conjH} implies Conjecture \ref{conjB}. 

We have: 

\begin{prop} 
Conjecture \ref{conjH} holds for monomial ideals.
\end{prop}

\begin{proof} Let $I$ be a CS monomial ideal and let $J$ be its generic initial ideal. Since $J$ is an initial ideal of $I$ (after a change of coordinates) we have:
$$\dim_K H^i_{\mm}(S/I)_a\leq \dim_K H^i_{\mm}(S/J)_a$$
for all $i$ and $a\in \ZZ^n$. 
The Alexander duals $I^*$ and $J^*$  of $I$ and $J$  are CS$^*$ ideals with the same $\ZZ^n$-graded  Hilbert function. Hence they have the same $\ZZ^n$-graded Betti numbers \cite[Proposition 1.9]{CDG2}.  In particular, $I^*$ and $J^*$ have the same $\ZZ$-graded Betti numbers. Then we deduce from Lemma~\ref{KohjiF} below that    
$$\dim_K H^i_{\mm}(S/I)_j=\dim_K H^i_{\mm}(S/J)_j$$
for all $j\in \ZZ$ and all $i\geq 0$ that, in combination with the inequality above, implies the desired equality. 
\end{proof}

\begin{lemma}\label{KohjiF} 
Let $I$ be a squarefree monomial ideal in a polynomial ring $R=K[x_1,\dots,x_N]$. Denote by $I^*$ its Alexander dual. Then for every $i\geq 0$ and every $j>0$ one has 
$$\dim_K H^i_m(R/I)_{-j}= \sum_{v=1}^{\min(i,j)}  \binom{j-1}{v-1} \beta_{i-v, N-v}(I^*)$$ 
while  $H^i_m(R/I)_{j}=0$ for $j>0$ and  $\dim_K H^i_m(R/I)_{0}= \beta_{i, N}(I^*)$. 
\end{lemma}

\begin{proof} Combining Hochster's formulas for Betti numbers (in the dual form) \cite[Corollary 1.40]{MS} and for local cohomology \cite[Theorem 13.13]{MS} one has that   
$$\dim_K H_\mm^i(R/I)_{-a}=\beta_{i-|a|, \one-a}(I^*)$$ for every $a\in \{0,1\}^N$. Computing the dimension of the $\ZZ$-graded component of $H_\mm^i(R/I)$ as sum of the corresponding multigraded components and taking into account Hochster's formula for local cohomology, one obtain the desired result. 
\end{proof}  

Let us conclude the section by stating a very general conjecture which is due to J\"urgen Herzog. Morally speaking Herzog's conjecture asserts  that a radical initial ideal behaves (homologically)  as the generic initial ideal with respect to the revlex order does. The conjecture, in  various forms,  has been discussed in several occasions by Herzog and his collaborators but, as far as we know, never appeared in print. A special case of it appears in Varbaro's PhD thesis as  Question 2.1.16 \cite{MV}.  Indeed, our conjecture \ref{conjB} is a special case of Herzog's conjecture.

\begin{conjecture}\label{conjJH} (Herzog) Let $I$ be a homogeneous ideal in a polynomial ring and $J$ an initial ideal of $I$ with respect to a term order. Assume $J$ is radical.  Then  $I$ and $J$ have the same extremal Betti numbers. In particular, $I$ and $J$ have the same projective dimension and regularity. 
\end{conjecture}

 Herzog's conjecture is known to be true for toric ideals (in toric coordinates) because in that case $J$ defines a Cohen-Macaulay  ring. Furthermore it is known in few other cases as,  for example,  homogeneous Cohen-Macaulay ASL with discrete Buchsbaum counterpart \cite[Thm.4.4]{MaM}.

\section{Binomial edge ideals}
\label{S2}

In this section we prove that every binomial edge ideal is a Cartwright-Sturmfels ideal. 

Let  $K$ be a field, let $S=K[x_1,\dots, x_n, y_1,\dots, y_n]$ and let $X$ be the  $2\times n$ matrix of variables
$$X=\left(
\begin{array}{cccc}
x_1 & x_2 & \cdots & x_n \\
y_1 & y_2 & \cdots & y_n 
\end{array}
\right)
.$$
Denote by $\Delta_{ij}$ the $2$-minor of $X$ corresponding to the column indices $i,j$, i.e.,
$$\Delta_{ij}=  \left | 
\begin{array}{cc}
x_i & x_j   \\
y_i & y_j  
\end{array}
\right |= x_iy_j-x_jy_i.
$$

Let $G$ be a graph on the vertex set $[n]=\{1,\dots, n\}$ and let 
$$J_G=(  \Delta_{ij} : \{i,j\} \mbox{ is an edge of } G ).$$
The ideal $J_G$ is called the binomial edge ideal of $G$. Binomial edge ideals are just ideals generated by subsets of the $2$-minors of $X$.  Herzog, Hibi, Hreinsdottir, Kahle, and Rauh in \cite{HHHKR} and, independently,  Ohtani in \cite{O} proved that  $J_G$ is radical. 
We will show that $J_G$ is a Cartwright-Sturmfels ideal with respect to the natural $\ZZ^n$-graded structure induced by $\deg(x_i)=\deg(y_i)=e_i\in \ZZ^n$.

\begin{thm}
\label{mainbinedge}
The $\ZZ^n$-graded generic initial ideal of $J_G$ is generated by the monomials $y_{a_1}\cdots y_{a_v}x_ix_j$ 
where $i, a_1,  \cdots,  a_v, j$ is a path in $G$. In particular $J_G$ is a Cartwright-Sturmfels ideal and therefore  all the initial ideals of $J_G$ are radical and $\reg(J_G)\leq n$. 
\end{thm} 

Here by a path of $G$ we mean a sequence of vertices without repetitions such that every pair of adjacent vertices form an  edge of the graph.  Note that in the description of the generators of the generic initial ideal one can assume that $i<j$ and that the path is minimal in the sense that  the only edges among the vertices $i, a_1,  \cdots,  a_v, j$ are $(i,a_1), (a_1,a_2), \dots, (a_v,j)$. That $\reg(J_G)\leq n$ has been proved originally by  Matsuda and Murai  in \cite{MM}, where they also conjectured that equality holds if and only if $G$ is a path of length $n-1$. In \cite{BBS} a universal Gr\"obner basis for $J_G$ is described and this implies that all the initial ideals of $J_G$,  in the given coordinates,  are radical. 

\begin{proof} Consider any term order such that $x_i>y_i$ for all $i$. 
To compute the generic initial ideal we first apply a multigraded upper triangular transformation $\phi$  to $J_G$, i.e. for every $i$ we have $\phi(x_i)=x_i$ and  $\phi(y_i)=\alpha_ix_i+y_i$ with $\alpha_i\in K$.  We obtain a matrix 
$$\phi(X)=\left(
\begin{array}{cccc}
x_1 & x_2 & \cdots & x_n \\
\alpha_1x_1+y_1 & \alpha_2x_2+y_2 & \cdots & \alpha_nx_n+y_n 
\end{array}
\right)
$$
whose $2$-minors are: 
$$\phi(\Delta_{ij})=  \left | 
\begin{array}{cc}
x_i & x_j   \\
\alpha_ix_i+y_i & \alpha_jx_j+y_j  
\end{array}
\right |= (\alpha_j-\alpha_i)x_ix_j+\Delta_{ij}.
$$
Assume that $\alpha_j\neq \alpha_i$ for $i\neq j$. We multiply  $\phi(\Delta_{ij})$ by the inverse of  $(\alpha_j-\alpha_i)$ and obtain: 
$$F_{ij}=x_ix_j-\lambda_{ij}\Delta_{ij}$$
with 
$$\lambda_{ij}=(\alpha_i-\alpha_j)^{-1}$$
so that $F_{ij}$ is monic. 
For later application we observe the following. For indices $1\leq i<j<k\leq n$ we consider the
$S$-polynomial $S(F_{ik}, F_{jk})$. Expanding   $S(F_{ik}, F_{jk})$  we have 
$$S(F_{ik}, F_{jk})=x_jF_{ik}-x_iF_{jk}=-\lambda_{jk}y_jx_ix_k+\lambda_{ik}y_ix_jx_k+(\lambda_{jk}-\lambda_{ik})y_kx_ix_j$$ 
 Now we perform division with reminder using $F_{ik}, F_{jk}$ and $F_{ij}$ and we have: 

$$S(F_{ik}, F_{jk})=-\lambda_{jk}y_jF_{ik}+\lambda_{ik}y_iF_{jk}+
(\lambda_{jk}-\lambda_{ik})y_kF_{ij}+r.$$
The remainder $r$ is 
$$r=-\lambda_{jk}y_j\lambda_{ik}\Delta_{ik}+\lambda_{ik}y_i\lambda_{jk}\Delta_{jk}+(\lambda_{jk}-\lambda_{ik})y_k\lambda_{ij}\Delta_{ij}$$
that is 
$$r=\lambda_{jk}\lambda_{ik}(-y_j\Delta_{ik}+ y_i \Delta_{jk}) +(\lambda_{jk}-\lambda_{ik})y_k\lambda_{ij}\Delta_{ij}.$$
Using the syzygy among minors:  
$$y_i\Delta_{jk}-y_j\Delta_{ik}+y_k\Delta_{ij}=0$$ 
we have 
$$r=\lambda_{jk}\lambda_{ik}(-y_k\Delta_{ij}) +(\lambda_{jk}-\lambda_{ik})y_k\lambda_{ij}\Delta_{ij}=
(-\lambda_{jk}\lambda_{ik} +\lambda_{jk}\lambda_{ij}-\lambda_{ik}\lambda_{ij})y_k\Delta_{ij}$$
and 

$$-\lambda_{jk}\lambda_{ik} +\lambda_{jk}\lambda_{ij}-\lambda_{ik}\lambda_{ij}=0$$ as can be checked by direct computation. Hence $r=0$ and we have a division with remainder $0$: 

\begin{equation}\label{eq*} S(F_{ik}, F_{jk})=-\lambda_{jk}y_jF_{ik}+\lambda_{ik}y_iF_{jk}+
(\lambda_{jk}-\lambda_{ik})y_kF_{ij}.\end{equation}

Now  we return to the ideal $J_G$ and its image under $\phi$: 

$$\phi(J_G)=(  F_{ij} : \{i,j\} \mbox{ is an edge of } G ). $$
Set 
$$F=\{  y_{a}F_{ij} :   i, a_1,  \cdots,  a_v, j \mbox{  is a path in } G\}$$ 
where
$$y_{a}=y_{a_1}\cdots y_{a_v}.$$
It is enough to  prove  that $F$  is a Gr\"obner basis for $\phi(J_G)$ for every $\phi$ such that $\alpha_j\neq \alpha_i$ for $i\neq j$.   We first observe that  $F\subset \phi(J_G)$, i.e. $y_{a}F_{ij} \in \phi(J_G)$ for every path   $i, a_1,  \cdots,  a_v, j$ in $G$. Since $F_{ij}$ and $\phi(\Delta_{ij})$ differ only by a non-zero scalar we may as well prove that $y_{a}\phi(\Delta_{ij}) \in \phi(J_G)$ for every path   $i, a_1,  \cdots,  a_v, j$ in $G$. 
 This is proved easily by induction on $v$, the case $v=0$ being trivial, applying to the matrix   $\phi(X)$ the following relation 
 $$(z_{1i}, z_{2i})\Delta_{jk}(Z) \subseteq  ( \Delta_{ij}(Z),  \Delta_{ik}(Z) )$$
 that holds for every $2\times n$ matrix $Z=(z_{ij})$ and every triplet of column indices $i,j,k$. 
 To prove that $F$ is a Gr\"obner basis we take two elements $y_{a}F_{ij}$ and $y_{b}F_{hk}$ in $F$ and prove that the corresponding $S$-polynomial reduces to $0$ via $F$. Here $a=a_1,\dots,a_v$ and $b=b_1\dots, b_r$ and $i,a,j$ and $h,b,k$ are  paths  in $G$.   We distinguish three cases: 
 
 Case 1). If $\{i,j\}=\{h,k\}$, we may assume $i=h$ and $j=k$. The corresponding $S$-polynomial is $0$. 
 
 Case 2). If $\{i,j\}\cap \{h,k\}=\emptyset$. Let $u=\GCD(y_a, y_b)$. Then $y_aF_{ij}=u (y_a/u)F_{ij}$ and 
 $y_bF_{hk}=u (y_b/u)F_{hk}$. Note that $(y_a/u)F_{ij}$ and $(y_b/u)F_{hk}$ have coprime leading terms and hence they form a Gr\"obner basis.  If a Gr\"obner basis is multiplied with a single polynomial the resulting set of polynomials is still a Gr\"obner basis. Hence $\{y_aF_{ij}, y_bF_{hk} \}$ is  a Gr\"obner  basis and the  $S$-polynomial of $y_aF_{ij}, y_bF_{hk}$ reduces to $0$ using only $y_aF_{ij}, y_bF_{hk}$. 
 
Case 3). If $\# \{i,j\}\cap \{h,k\}=1$. Renaming the column indices  we may assume that $i=1$, $h=2$ and $j=k=n$. Hence we deal with $y_aF_{1n}$ and $y_bF_{2n}$. Let $u=\LCM(y_a,y_b)$. We have:
$$S(y_aF_{1n}, y_bF_{2n})=uS(F_{1n}, F_{2n}).$$
Considering~(\ref{eq*}) with $i=1$, $j=2$ and $k=n$ and multiplying both sides with $u$  we have: 

\begin{equation}\label{eq**}\quad S(y_aF_{1n}, y_bF_{2n})=-\lambda_{2n}y_2uF_{1n}+\lambda_{1n}y_1uF_{2n}+
(\lambda_{2n}-\lambda_{1n})y_nuF_{12}.\end{equation}
Since~(\ref{eq*}) is  a division with reminder $0$ of  $S(F_{1n}, F_{2n})$ with respect to $F_{1n}, F_{2n}$ and $F_{12}$ we may conclude that~(\ref{eq**}) is a division with reminder $0$ of  $S(y_aF_{1n}, y_bF_{2n})$ with respect to the set $F$ if we prove that $y_2uF_{1n}, y_1uF_{2n}$ and $y_nuF_{12}$ are multiples of elements of $F$.  
Clearly $y_2uF_{1n}$ is a monomial multiple of $y_aF_{1n}$ and $y_1uF_{2n}$ is a monomial multiple of $y_bF_{1n}$. So we are left with $y_nuF_{12}$. 
If $u$ is divisible by a monomial $y_d=y_{d_1}\cdots y_{d_t}$ such that $1,d_1,\dots,d_t,2$ is a path in $G$ then  $y_nuF_{12}$   is multiple of the element $y_dF_{12}$ of $F$. On the other hand, if $u$ is not divisible  by a monomial $y_d=y_{d_1}\cdots, y_{d_t}$ such that $1,d_1,\dots,d_t,2$  is a path in $G$ then 
$$\{1,a_1,\dots, a_v\} \cap \{2,b_1,\dots, b_r\} =\emptyset \mbox{ and } u=y_ay_b.$$  
In this case $1,a,n,b,2$ is a path from $1$ to $2$ in $G$ and hence $y_nuF_{12}=y_ny_ay_bF_{12}$ is indeed in $F$.

This concludes the proof that the set $F$ is a Gr\"obner basis. The remaining statements follow from general facts on Cartwright-Sturmfels ideals established in \cite[Remark 1.5, Corollary 1.15]{CDG2}. 
\end{proof}

We describe now the minimal primes of the generic initial ideal of $J_G$. We denote by $c(G)$ the number of connected components of a graph $G$. 
For a subset $T$ of $[n]$ let $G_T$ be the restriction of $G$ to $T$ and set 
$$U_T=( x_ix_j : i,j\in T \mbox{ and are connected by a path in } G_T)+\sum_{i\not\in T}(x_i, y_i).$$
It is clear that $\gin(J_G)\subseteq U_T$ for every $T$.
Furthermore let $E$ be a subset of $T$ such that $E$ contains exactly one vertex for each connected component of $G_T$
 and set 
$$U_{T,E}=( x_i  : i\in T\setminus E)+\sum_{i\not\in T}(x_i, y_i).$$
 Then: 
\begin{prop} The minimal primes of $\gin(J_G)$ are exactly the ideals $U_{T,E}$ where $T$ is chosen so that for every $i\in [n]\setminus T$ one has $c(G_{T\cup \{i\}})<c(G_T)$.   
\end{prop} 

\begin{proof} Let $P$ be a minimal prime of $\gin(J_G)$ and let $T=\{ i \in [n] : y_i \not\in P\}$. Since $P$ is Borel fixed then it follows from Theorem \ref{mainbinedge}   that  $U_T\subseteq P$ and so it follows that $\gin(J_G)=\bigcap_T U_T$. Now, $U_T\subseteq  U_{T_1}$ if and only if $T_1\subseteq T$ and $c(G_{T_1})=c(G_T)$. So it follows that $\gin(J_G)=\bigcap_T U_T$ where the intersection is restricted to the $T$  such that for every $i\in [n]\setminus T$ one has $c(G_{T\cup \{i\}})<c(G_T)$. Finally, observe that $U_T=\bigcap U_{T,E}$ where the intersection ranges over $E\subset T$ such that  $E$ contains exactly one vertex for each connected component of $G_T$. 
\end{proof}

In \cite{SZ} Schenzel and Zafar computed the structure of the local cohomology modules (indeed of the corresponding Ext-modules)   of the   binomial edge ideal associated to a complete bipartite graph.  These results might shed some light on Conjecture \ref{conjH} and might suggest more precise versions of it.

\section{Closure of linear spaces in products of projective spaces} 
\label{S3}

Let $T=K[x_1,\dots, x_n]$ be a polynomial ring with a standard $\ZZ^u$-graded  structure, i.e. $\deg(x_i)\in \{e_1,\dots, e_u\}$ for every $i$.  Let $S=T[y_1,\dots,y_u]$ with the $\ZZ^u$-graded structure obtained  by extending that of  $T$ by letting $\deg y_i=e_i\in \ZZ^u$. Given  $f=\sum_{i=1}^v \lambda_i x^{a_i}\in T\setminus \{0\}$ we consider its $\ZZ^u$-homogenization $f^{\hom}\in S=T[y_1,\dots,y_u]$ defined as $$f^{\hom}=\sum_{i=1}^v \lambda_i x^{a_i}y^{b-b_i}$$ 
where $\deg x^{a_i}=b_i\in \ZZ^u$ and $y^b=\LCM(y^{b_1},\dots, y^{b_v} )$.
For any $c\in \ZZ^u$ such that $\LCM(y^{b_1},\dots, y^{b_v} ) | y^c$, we define 
$$f^{\hom,c}=\sum  \lambda_{i} x^{a_i}y^{c-b_i}.$$
By construction, $f^{\hom}$ is $\ZZ^u$-homogeneous of degree $b$ and $f^{\hom,c}=f^{\hom}y^{c-b}$. Given an ideal $I$ of $T$ its $\ZZ^u$-homogenization is defined as 
$$I^{\hom}=( f^{\hom} : f\in I\setminus\{0\} )\subset S$$
and it is clearly a $\ZZ^u$-graded ideal of $S$. For generalities about homogenization of ideals we refer the reader to \cite{KR}. Here we just recall that if $I=(f_1,\dots,f_t)$ then 
$$I^{\hom}=(f_1^{\hom},\dots,f_t^{\hom}):(\prod_{i=1}^u y_i)^{\infty},$$ see \cite[Corollary 4.3.8]{KR}. Let $c\in \ZZ^u$ such that for every $i=1,\dots, t$ and for every monomial $x^a$ in the support of $f_i$ we have $y^v | y^c$, where $v=\deg x^a\in \ZZ^u$. Then we have 
$$I^{\hom}=(f_1^{\hom,c},\dots,f_t^{\hom,c}):(\prod_{i=1}^u y_i)^{\infty},$$
because $f_i^{\hom,c}$ and  $f_i^{\hom}$ differ only by a monomial in the $y$'s.   

We denote by $I^\star$ the largest $\ZZ^u$-graded ideal of $T$ contained in $I$,  i.e. the ideal generated by the $\ZZ^u$-graded elements in $I$. We show that: 

\begin{thm}\label{homoline} 
Let $T=K[x_1,\dots, x_n]$ be a polynomial ring with a standard $\ZZ^u$-graded  structure. Let  $V$ be a 
vector space of linear forms of $T$ (i.e. elements of  degree $1$ with respect to the standard $\ZZ$-graded structure) and $J(V)$ be the ideal generated by $V$.  Then 
\begin{itemize}  
\item[(1)]  $J(V)^{\hom}$  and $J(V)^\star$ are Cartwright-Sturmfels ideals.   
\item[(2)] Both   $J(V)^{\hom}$  and $J(V)^\star$  define  Cohen-Macaulay normal rings.  
\end{itemize} 

\end{thm} 

If $K$ is algebraically closed, $T$ is equipped with the natural $\ZZ^n$-graded structure, and $L$ is the zero locus $V$ in $\AAA_K^n$, then ideal $J(V)^{\hom}$ is exactly the defining ideal of the closure $\tilde{L}$ of $L$ in $\PP^{1}\times \dots \times \PP^{1}$,  i.e. $J(V)^{\hom}=I(\tilde{L})$ in the notation of Ardila and Boocher.  
If instead we equip $T$ with a  $\ZZ^u$-graded structure where $a_i$ variables have degree $e_i$, then the ideal $J(V)^{\hom}$ is the ideal associated to the closure of $L$ in the  product $\PP^{a_1}\times \dots \times \PP^{a_u}$, i.e. $J(V)^{\hom}$ is the ideal denoted by $I_a(\tilde{L})$ in the introduction. 

\begin{proof} (1) 
The assertion on  $J(V)^\star$ follows from the one on $J(V)^{\hom}$  and \cite[Theorem 1.16]{CDG2} since,    by \cite[Tutorial 50] {KR},  one has  $$J(V)^\star=J(V)^{\hom}\cap T.$$
To prove the assertion for $J(V)^{\hom}$ we argue as follows. For a matrix $X$ and an integer $t$ we denote by $I_t(X)$ the ideal generated by the $t$-minors of $X$. Let $\ell$ be a linear form of $T$, say  $\ell=\sum_{i=1}^u \ell_i$ where $\ell_i$ is $\ZZ^u$-homogeneous of degree $e_i\in \ZZ^u$.  Set $\one=\sum_{1}^u e_i$ and notice that $\ell^{\hom,\one}= \prod_{j=1}^u  y_j \sum_{i=1}^u \ell_i/y_i$  can be written as 
$$H^{\hom,\one}=\det 
\left(\begin{array}{ccccccccccc} 
  y_1 & 0 & \cdots   &   \cdots &   0 & \ell_1     \\
 -y_2& y_2 & 0 &   \cdots &   0 & \ell_2     \\
 0     & -y_3& y_3 &     \cdots &   0 & \ell_3    \\
\vdots&\vdots &  \vdots & \ddots&\vdots &  \vdots   & \\
\vdots&\vdots &  \vdots & \ddots& y_{u-1} &  \vdots   & \\
 0     & 0& \cdots & \cdots   &-y_u & \ell_u    
\end{array}
\right).
$$
Now, if $V=\langle L_1,\dots, L_v\rangle$, let $X_{[u]}$ be the $u\times (v+u-1)$ matrix with block decomposition 
$$X_{[u]}=(Y_{[u]}\ |\  M_{[u]} )$$ 
where 
$$Y_{[u]}=
\left(\begin{array}{ccccccccccc} 
  y_1 & 0 & \cdots   &   \cdots &   0      \\
 -y_2& y_2 & 0 &   \cdots &   0      \\
 0     & -y_3& y_3 &     \cdots &   0     \\
\vdots&\vdots &  \vdots & \ddots&\vdots &   \\
\vdots&\vdots &  \vdots & \ddots& y_{u-1} &   \\
 0     & 0& \cdots & \cdots   &-y_u 
\end{array}
\right)
$$
and 
$$M_{[u]}=\left(\begin{array}{cccccccccccccc} 
L_{11}  & \dots &L_{1v}    \\
L_{21}  & \dots &L_{2v}   \\
L_{31}   & \dots &L_{3v}    \\
\vdots&\vdots &  \vdots     \\
\vdots&\vdots &  \vdots    \\
L_{u1}  & \dots &L_{uv}    
\end{array}
\right)$$
is the $u\times v$ matrix whose $i$-th column is  given by the $\ZZ^u$-homogeneous components of $L_i$, that is $L_i=\sum_{q=1}^u  L_{qi}$. Let  $H_{[u]}$ be the ideal generated  by the $u$-minors $\Delta_1 ,\dots, \Delta_v$ of $X_{[u]}$, where $\Delta_i$  involves the $(u-1)$ columns of $Y_{[u]}$ and the $i$-th  column of $M_{[u]}$. By construction $H_{[u]} =(L_1^{\hom,\one},  \dots, L_v^{\hom,\one})$, hence $J(V)^{\hom}=H_{[u]} : (\prod_{i=1}^u y_i)^{\infty}$.   
It follows immediately from the  straightening law for minors (see \cite[Section 4]{BV}) that $I_{u-1}(Y_{[u]} )I_u(X_{[u]} )\subseteq H_{[u]} $ and obviously $H_{[u]} \subseteq I_u(X_{[u]} )$. In this case   $I_{u-1}(Y_{[u]} )$ is generated by the squarefree monomials of degree $u-1$ in the variables $y_1,\dots,y_u$. Hence  
 %\begin{eqnarray*}
$$ I_u(X_{[u]} ) : (\prod_{i=1}^u y_i)^{\infty}\supseteq 
 H_{[u]} : (\prod_{i=1}^u y_i)^{\infty}\supseteq 
 \left(  I_{u-1}(Y_{[u]} )I_u(X_{[u]} ) \right) : (\prod_{i=1}^u y_i)^{\infty}= %\\
  I_u(X_{[u]}  ) : (\prod_{i=1}^u y_i)^{\infty}.$$
%  \end{eqnarray*}
Summing up, we have shown that
$$J(V)^{\hom}=H_{[u]} : (\prod_{i=1}^u y_i)^{\infty}= I_u(X_{[u]} ) : (\prod_{i=1}^u y_i)^{\infty}.$$
The matrix $X_{[u]} $ is row-graded, i.e. the entries in its $i$-th row are homogeneous of degree $e_i\in \ZZ^u$. Hence by \cite[Corollary 1.19]{CDG2} its ideal of maximal minors is a Cartwright-Sturmfels ideal. In particular $I_u(X_{[u]} )$ is radical, hence $$J(V)^{\hom}= I_u(X_{[u]} ) : (\prod_{i=1}^u y_i).$$  
By \cite[Theorem 1.16]{CDG2} it follows that $J(V)^{\hom}$  is a Cartwright-Sturmfels ideal as well. This concludes the proof of (1). 

(2) Since $J(V)$ is a prime ideal, then  $J(V)^{\hom}$ is prime (see  e.g.~ \cite[Proposition 4.3.10]{KR}). Then $J(V)^\star=J(V)^{\hom}\cap T$ is prime as well. One can easily check that the ideals $J(V)^{\hom}$  and $J(V)^\star$ are geometrically primes, i.e. they remain prime under field extensions.  Hence we may assume without loss of generality that $K$ is algebraically closed. So we may apply Brion's Theorem \ref{Brion} and conclude that both $J(V)^{\hom}$  and $J(V)^\star$  define  Cohen-Macaulay normal rings.   
\end{proof}
 
In order to identify  generators of  $J(V)^{\hom}$, we proceed as follows. For every non-empty subset $A$ of $\{1,\dots,u\}$ let $$V_A=V\cap \oplus_{i\in A} T_{e_i}$$  
and consider the ideal $J(V_A)$ generated by $V_A$. We associate to $J(V_A)$  the ideal $H_A$ generated the homogenization $L^{\hom,c}$ of the generators of $J(V_A)$ with respect to the vector $c=\sum_{i\in A} e_i$ and the corresponding matrices $X_A,Y_A, M_A$ constructed as in the proof of Theorem~\ref{homoline}. In the proof of Theorem~\ref{homoline} we showed that 
$$J(V_A)^{\hom}=I_{|A|}(X_A): (\prod_{i \in A}  y_i) =I_{|A|}(X_A): (\prod_{i=1}^u  y_i) $$
and, since $J(V_A) \subseteq  J(V)$, we obtain $I_{|A|}(X_A)\subseteq  J(V_A)^{\hom}\subseteq J(V)^{\hom}$. Therefore we have
$$\sum_{A\neq \emptyset}  I_{|A|}(X_A)\subseteq   J(V)^{\hom}.$$ 
We claim that equality holds. In order to prove our claim, we will need the following:

\begin{lemma}\label{techle1}
Let $J$ be a  $\ZZ^u$-graded Cartwright-Sturmfels ideal, let $F$ be a product of  $\ZZ^u$-graded linear forms. Let $J_1$ be the ideal generated by the elements of $J:(F)$ of degree smaller than $(1,\dots,1)$. Then $J:(F)=J+J_1$. 
\end{lemma}

\begin{proof}
By induction on the degree of $F$ and by \cite[Theorem 1.16]{CDG2}, we may assume that $F$ is a $\ZZ^u$-graded linear form, say of degree $e_u$. After a change of coordinates we may also assume that $F$ is a variable, call it $x$. We introduce a revlex term order $<$ such that $x$ is the smallest variable with respect to $<$. Let $G_1,\dots,G_v$ be a Gr\"obner basis of $J$ with respect to $<$.   Since $J$  is a Cartwright-Sturmfels ideal the $G_i$'s have degrees smaller than or equal to $(1,\dots,1)$. Some of them, say $G_1,\dots,G_w$, have a leading term divisible by $x$ and the remaining $G_{w+1},\dots,G_v$ do not. Hence $G_j=xH_j$ for  $j=1,\dots, w$. It follows that $J:x=(H_1,\dots,H_w, G_{w+1},\dots,G_v)$. Hence $J_1=(H_1,\dots ,H_w)+( G_j : \deg G_j<(1,\dots,1) )$ and  $J:x=J+J_1$. 
\end{proof} 
 
\begin{lemma}\label{techle2}
With the notation introduced above, let $F\in T$ be a $\ZZ^u$-graded polynomial of degree $a\in \ZZ^u$  with $a_u=0$.  
Assume that  $Fy_u \in H_{[u]}$, then $F\in H_A$  with  $A=[u]\setminus {u}$. 
\end{lemma} 

\begin{proof}  
Let $\pi_u: \oplus_{i=1}^u  T_{e_i} \to T_{e_u}$ be the projection on the $u$-th homogeneous component. We may choose a basis of $V$ of the form $L_1,\dots, L_h, U_1,\dots, U_k$ so that   $\pi_u(L_1),\dots, \pi_u(L_h)$ form a $K$-basis of $\pi_u(V)$ and $\pi_u(U_i)=0$ for every $i$. By construction, $H_{[u]}$ is generated by the homogenization with respect to the vector $\one$ of $L_1,\dots, L_h,U_1,\dots, U_k$.
Notice that for every $i=1,\dots h$ one has that $L_i^{\hom, \one}= W_i(y_1\cdots y_{u-1})+y_uW_i'$, where $W_i=\pi_u(L_i)$ and 
$W_i'$ is homogeneous of degree $\one-e_u$. Moreover, the homogenization with respect to the vector $\one$ of $U_1,\dots, U_k$ generates $y_uH_A$, with $A=[u]-\{u\}$.  Since $Fy_u$ is in $H_{[u]}$, then
$$Fy_u=\sum_{i=1}^h E_i  (W_iy_i\cdots y_{u-1}+y_uW_i' )+ y_uC $$
where $C\in H_A$ and the $E_i$'s are $\ZZ^u$-homogeneous with the homogeneous component of degree $e_u$ equal to $0$. 
 Since $W_1,\dots ,W_h, y_u$ are linearly independent elements of degree $e_u$, it follows that $E_i=0$ for every $i$ and $F=C\in J_A$. 
\end{proof}

We can now give an explicit description of $J(V)^{\hom}$ and $J(V)^{\star}$ as sums of ideals of minors.
\begin{thm}\label{determinant}
With the notations above one has:  
$$J(V)^{\hom}= \sum_{A\neq \emptyset}  I_{|A|}(X_A)$$ 
and 
$$J(V)^{\star}= \sum_{A\neq \emptyset}  I_{|A|}(M_A).$$ 
\end{thm} 

\begin{proof} 
In order to prove the first statement, by the proof of Theorem \ref{homoline} it suffices to show that
$$I_{u}(X_{[u]}):(\prod_{i=1}^u y_i)=\sum_{A\neq \emptyset}  I_{|A|}(X_A).$$
By induction, it suffices to prove that
$$I_{u}(X_{[u]}):(\prod_{i=1}^u y_i)= I_{u}(X_{[u]})+W$$
where 
$$W=\sum_{j=1}^u \left(  I_{u-1}(X_{[u]\setminus \{j\}}):(\prod_{i=1}^u y_i)  \right).$$
 
By Lemma~\ref{techle1} it suffices to show that any $\ZZ^u$-graded element $G\in I_{u}(X_{[u]}):(\prod_{i=1}^u y_i)$ of degree smaller than $(1,\dots,1)$ is in $W$. We may assume that $G$ has degree $0$ in the $u$-th coordinate. 
In the proof of Theorem~\ref{homoline} we observed that every squarefree monomial of degree $u-1$ in the $y_i$'s multiplies $I_u(X_{[u]})$ in the ideal $H_{[u]} =(L_1^{\hom,\one},  \dots, L_v^{\hom,\one})$. Since by assumption $Gy_1\cdots y_u\in  I_u(X_{[u]})$, then 
$$G(y_1\cdots y_{u-1})^2y_u \in  H_{[u]}.$$
Notice that the polynomial $F=G(y_1\cdots y_{u-1})^2$ is $\ZZ^u$-graded and has degree $0$ in the last coordinate and $Fy_u\in H_{[u]}$.  It follows from Lemma~\ref{techle2} that  $F\in H_A$ with $A=[u]\setminus\{u\}$. Hence $G\in H_A:(\prod_{i=1}^u y_i)^\infty=I_{|A|}(X_A):(\prod_{i=1}^u y_i)\subseteq  W$.  

The second statement may be deduced from the first as follows. One observes that   $J(V)^{\hom}$ is homogeneous with respect to the  $\ZZ$-graded structure induced by assigning degree $1$ to the $y_i$'s and degree $0$ to the elements of $T$. Since  $J(V)^{\star}=J(V)^{\hom} \cap T$, then $J(V)^{\star}$ is obtained from $J(V)^{\hom}$ by setting to $0$ the $y_i$'s. 
\end{proof}

In the next example we illustrate Theorems \ref{homoline} and \ref{determinant} and their proofs.
We consider  the linear space discussed in \cite[Example 1.7]{AB} and we homogenize with respect to a different multigrading.
\begin{ex}
\label{esempio}
Let $T=K[x_1,\ldots,x_6]$ with the $\ZZ^3$-graded structure induced by 
$$\deg(x_i)=\left\{\begin{array}{ll} 
e_1 &  \mbox{for }i=1,2,\\
e_2 &  \mbox{for }i=3,\\
e_3 &  \mbox{for }i=4,5,6.
\end{array}
\right.
$$
Let $V=\langle L_1,L_2,L_3 \rangle$ with $L_1=x_1+x_2+x_6,\ L_2=x_2-x_3+x_5,\ L_3=x_3+x_4$. 

So one has $u=v=3$ and
$$X_{\{1,2,3\}}=\left(
\begin{array}{ccccc}
y_1 & 0 & x_1+x_2 & x_2 & 0 \\
-y_2 & y_2 & 0 & -x_3& x_3 \\
0 & -y_3 & x_6 & x_5  & x_4
\end{array}
\right).
$$
%Then the ideal $H_{[3]}$ is generated by $\Delta_1, \Delta_2, \Delta_3$ with
%$$\Delta_1=\left|
%\begin{array}{ccc}
%y_1 & 0 & x_1+x_2  \\
%-y_2 & y_2 & 0  \\
%0 & -y_3 & x_6 
%\end{array}
%\right|,\ 
%\Delta_2=\left|
%\begin{array}{ccc}
%y_1 & 0 & x_2 \\
%-y_2 & y_2 & -x_3\\
%0 & -y_3 & x_5 
%\end{array}
%\right|,
%\Delta_3=\left|
%\begin{array}{ccccc}
%y_1 & 0 &  0 \\
%-y_2 & y_2 &  x_3 \\
%0 & -y_3 & x_4
%\end{array}
%\right|,
%$$
%and by definition
%$$J(V)^{\hom}=(\Delta_1, \Delta_2,\Delta_3)
%%x_1y_2y_3+x_2y_2y_3+x_6y_1y_2, 
%%x_2y_2y_3-x_3y_1y_3+x_5y_1y_3, 
%%x_3y_1y_3+x_4y_1y_3)
%:(y_1y_2y_3)^{\infty}.$$
%Moreover one can verify using \cite{CoCoa_5} that 

In Theorem \ref{homoline} we proved that $J(V)^{\hom}$ is a CS ideal and $$J(V)^{\hom}=I_3(X_{\{1,2,3\}}):(y_1y_2y_3).$$ 
In order to obtain the generators of $J(V)^{\hom}$, we use Theorem \ref{determinant}.  
The relevant subsets $A\subseteq  \{1,2,3\}$ are those for which $V_A\neq \{0\}$, hence in this case they correspond to 
%where 
%$I_3(X_{[3]})=(\Delta_1, \Delta_2, \Delta_3, 
%x_3x_6y_1 +x_1x_5y_2 +x_2x_5y_2 -x_2x_6y_2, 
%-x_3x_6y_1 +x_1x_4y_2 +x_2x_4y_2, 
%-x_3x_4y_1 -x_3x_5y_1 +x_2x_4y_2, 
%-x_1x_5y_2 -x_2x_5y_2 +x_2x_6y_2 +x_1x_3y_3 +x_2x_3y_3, 
%-x_1x_4y_2 -x_2x_4y_2 -x_1x_3y_3 -x_2x_3y_3, -x_2x_4y_2 -x_2x_3y_3, 
%-x_1x_3x_4 -x_2x_3x_4 -x_1x_3x_5 -x_2x_3x_5 +x_2x_3x_6)$
%To write explicitly the generators of $J(V)^{\hom}$ we use Theorem \ref{determinant} with the description before Lemma \ref{techle1}. For $A=\{1\},\{2\},\{3\},\{1,2\}$ one has $V_A=\{0\}$. Moreover one has 
$$V_{\{1,3\}}=\langle L_1, L_2+L_3 \rangle, \quad  V_{\{2,3\}}=\langle L_3 \rangle$$ 
and of course $V_{\{1,2,3\}}=V$. 
The corresponding matrices are
$$
X_{\{1,3\}}=\left(
\begin{array}{ccccc}
y_1 &  x_1+x_2 &x_2 \\
 -y_3 & x_6 & x_4+x_5
\end{array}
\right), \ 
X_{\{2,3\}}=\left(
\begin{array}{ccccc}
y_2 & x_3 \\
-y_3 & x_4
\end{array}
\right)
$$
thus by Theorem \ref{determinant} we have
$$J(V)^{\hom}=I_2(X_{\{1,3\}})+I_2(X_{\{2,3\}})+I_3(X_{\{1,2,3\}}).$$
It turns out that the generators of $I_3(X_{\{1,2,3\}})$ are superfluous, so that
$$J(V)^{\hom} 
%$$((x_1+x_2)y_3+x_6y_1, x_2y_3+(x_4+x_5)y_1, (x_1+x_2)(x_4+x_5)-x_2x_6,
%x_3y_3+x_4y_2)+I_3(X_{[3]})=$$
=(x_4y_2 +x_3y_3,  \ x_6y_1 +x_1y_3 +x_2y_3, \  x_4y_1 +x_5y_1 +x_2y_3, \  x_1x_4 +x_2x_4 +x_1x_5 +x_2x_5 -x_2x_6).$$
Finally by Theorem \ref{determinant} we have $J(V)^\star=I_2(M_{\{1,3\}})+I_2(M_{\{2,3\}})+I_3(M_{\{1,2,3\}})$, with
$$
M_{\{1,3\}}=\left(
\begin{array}{cc}
  x_1+x_2 &x_2 \\
  x_6 & x_4+x_5
\end{array}
\right), \ 
M_{\{2,3\}}=\left(
\begin{array}{c}
 x_3 \\
 x_4
\end{array}
\right), \ 
M_{\{1,2,3\}}=\left(
\begin{array}{ccc}
 x_1+x_2 & x_2 & 0 \\
 0 & -x_3& x_3 \\
 x_6 & x_5  & x_4
\end{array}
\right),
$$
so that  one gets $J(V)^\star=(\det M_{\{1,3\}})$.
\end{ex}

In \cite{AB} Ardila and Boocher consider $T=K[x_1,\dots, x_n]$ with the standard $\ZZ^n$-graded structure  induced  by $\deg x_i=e_i\in \ZZ^n$.
In this setting they prove, among other things, that all the initial ideals of $J(V)^{\hom}$ (in the given coordinates) are squarefree. Moreover, the Betti numbers of $J(V)^{\hom}$ and $\inid(J(V)^{\hom})$ coincide. We can recover and generalize these results as follows.  

\begin{thm}\label{AB1}
Let $T=K[x_1,\dots, x_n]$ be a polynomial ring with the  standard $\ZZ^n$-graded  structure. Let  $V$ be a 
vector space of linear forms of $T$ (i.e. elements of  degree $1$ with respect to the standard $\ZZ$-graded structure) and let $J(V)$ be the ideal generated by $V$.  
Then  $J(V)^{\hom}\subset S=K[x_1,\dots, x_n, y_1,\dots, y_n]$ is a Cartwright-Sturmfels as well as a 
Cartwright-Sturmfels$^*$ ideal. Furthermore every ideal $H$ of $S$ with the same $\ZZ^n$-graded Hilbert function as $J(V)^{\hom}$ is radical, Cohen-Macaulay and $\beta_{i,a}(H)=\beta_{i,a}(J(V)^{\hom})$ for every $i\in \NN$ and $a\in \NN^n$. 
\end{thm} 

\begin{proof} 
We proved that  $J(V)^{\hom}$ is a Cartwright-Sturmfels ideal. Hence every ideal $H$ of  $S$ with the same $\ZZ^n$-graded Hilbert function as $J(V)^{\hom}$ is radical. Let $G=\gin(J(V)^{\hom})$ be the multigraded generic initial ideal of $J(V)^{\hom}$. Since  $J(V)^{\hom}$  is a $\ZZ^n$-graded prime ideal,  the largest variable of each block does not appear in the generators of  $G$. Since we have only two variables in each block it follows that the generators of $G$ involve only one variable per block, hence $J(V)^{\hom}$ is a Cartwright-Sturmfels$^*$ ideal. Hence every ideal $H$ of $S$ with the $\ZZ^n$-graded Hilbert function of $J(V)^{\hom}$  satisfies  $\beta_{i,a}(H)=\beta_{i,a}(J(V)^{\hom})$ for every $i\in \NN$ and $a\in \NN^n$.  The Cohen-Macaulay property follows from Brion's Theorem~\ref{Brion}.\end{proof} 
 
Ardila and Boocher in~\cite{AB} computed the multidegree of $S/J(V)^{\hom}$ in their setting. We are able to compute the multidegree in  general.
 
Let  $T=K[x_1,\ldots,x_n]$ with any $\ZZ^u$-graded structure. Let $V=\langle L_1,\ldots,L_v \rangle$ and consider the $v\times n$ matrix $M_V$, whose $(i,j)$-entry  is the coefficient of $x_j$ in $L_i$.  To $M_V$ we associate the  matroid  $\M_V$, whose elements are the subsets of $[n]$ corresponding to linearly independent columns of $M_V$. A basis of $\M_V$ is a maximal element, i.e., a set of column indices  $\{b_1,\ldots,b_v\}$ corresponding to a basis of the column space of $M_V$. 
To every basis $b=\{b_1,\ldots,b_v\}$ we associate a multidegree $\deg(b)=\deg(x_{b_1}\cdots x_{b_r})\in\ZZ^u$ and let
$$D_V=\{ \deg(b) :  b \mbox{ is a basis of }\M_V \}.$$ 
With this notation we have:

\begin{thm}\label{multideg}
Let  $R=S/J(V)^{\hom}$. The multidegree of $R$ is given by the formula: 
$$\MDeg_{R}(z)=\sum_{w \in D_V} z^{w}.$$ 
\end{thm}

\begin{proof}
If $W\subset\AAA^n$ is the affine $(n-v)$-dimensional linear space corresponding to $J(V)$, then $J(V)^{\hom}$ corresponds to the closure $\overline{W}$ of $W$ in the product of projective spaces $\PP^{n_1}\times\ldots\times\PP^{n_u}$, where $n_i=\dim_K T_{e_i}$, $1\leq i\leq u$. By the results we have recalled in Section \ref{S1} we have: 
$$\MDeg_{R}(z)=\sum \deg(H_{c}\cap\overline{W}) z^{c}$$ 
where the sum runs over the $c=(c_1,\ldots,c_u)\in \NN^u$  with $c_i\leq n_i$ and  $|c|=v$. Moreover, $H_{c}=W_1\times\ldots \times W_u$ and $W_i \subseteq\PP^{n_i}$ is a generic linear subspace with $\dim W_i=c_i$. Here $\deg(H_{c}\cap\overline{W}) $ denotes the usual intersection multiplicity of $H_{c}$ and $\overline{W}$.

We claim that the intersection of $H_{c}$ and $\overline{W}$ is affine (for a generic choice of $H_{c}$). In fact, since $\overline{W}$ is irreducible and not contained in the hyperplane $H_i$ of equation $y_i=0$ for any $i$, then $\dim(\overline{W}\cap H_i)=n-v-1$. Therefore, a generic  $H_{c}$ has empty intersection with $\overline{W}\cap H_i$, since $\dim H_{c}+\dim (\overline{W}\cap H_i)=n-1<n$. Finally, since there are $u$ hyperplanes $H_i$, the intersection of a generic $H_{c}$ with $\overline{W}$ avoids them all.

Since the intersection of $H_{c}$ and $\overline{W}$ is affine, then it  corresponds to the intersection of $W$ with $W_1\times\ldots\times W_u \subset \AAA^{n_1}\times\ldots\times\AAA^{n_u}=\AAA^n$  where $W_i$ is a generic linear subspace of $\AAA^{n_i}$ that contains the origin.   

Therefore, the defining equations of the affine part of $H_{c}$ are general elements  $\ell_{i,j}\in T_{e_i}$ with $1\leq i\leq u$ and $1\leq j \leq n_i-c_i$. 
In particular, $m(H_{c},\overline{W})\in\{0,1\}$ and it is $1$   if and only if  the $\ell_{i,j}$'s and the $L_i$'s are linearly independent.
The associated determinant is the linear combination the maximal minors of $M_V$ corresponding to bases of $\M_V$ of multidegree $c$ whose coefficients are generic. Hence $m(H_{c},\overline{W})=1$ if and only if $c\in D_V$. 

%In addition, up to a multigraded (affine) change of variables, we can assume that the $\ell_{i,j}$'s are variables of $T$. Notice in fact that, if $V'=\langle L'_1,\ldots,L'_v \rangle$ is obtained from $V$ by a multigraded change of variables and $M'$, $\M'$ are the associated matrix and matroid as above, then $\M$ has a basis in multidegree ${\mathbf b}\in\ZZ^u$ if and only if $\M'$ has a basis in multidegree ${\mathbf b}\in\ZZ^u$. Moreover, each basis $b$ of $\M$ corresponds to a set of variables $x_{b_1},\ldots,x_{b_u}$ in $T$ such that the linear space $H_b$ of equations $x_{b_1}=\ldots=x_{b_u}=0$ intersects $\overline{W}$ in a point.
\end{proof}

We illustrate Theorem \ref{multideg} by considering again Example \ref{esempio}.  

\begin{ex}
Let $T=K[x_1,\ldots,x_6]$ with the same $\ZZ^3$-graded structure as in Example \ref{esempio}, i.e., let
$$\deg(x_i)=\left\{\begin{array}{ll} 
e_1 &  \mbox{for }i=1,2,\\
e_2 &  \mbox{for }i=3,\\
e_3 &  \mbox{for }i=4,5,6.
\end{array}
\right.
$$
Let $V=\langle L_1,L_2,L_3 \rangle$ with $L_1=x_1+x_2+x_6,\ L_2=x_2-x_3+x_5,\ L_3=x_3+x_4$, and let $R=S/J(V)^{\hom}$. The matrix associated to $V$ is
$$M_V=
\left(\begin{array}{cccccc} 
1 & 1& 0& 0 & 0 & 1\\
0 & 1& -1& 0 & 1 & 0\\
0 & 0& 1& 1 & 0 & 0
\end{array}
\right).
$$
The set of bases of the matroid $\M_V$  is
$$\{123,124,134,135,145,234,235,236,245,246,346,356,456\}.$$
To every basis we associate a multidegree and a monomial in $K[z_1,z_2,z_3]$, for example $123$ corresponds to the degree $\deg(x_1x_2x_3)=(2,1,0)$ and to the monomial $z_1^2z_2$.
The set of the degrees of the bases of $\M_V$ is 
$$D_V=\{ (2,1,0), (2,0,1), (1,1,1), (1,0,2), (0,1,2), (0,0,3) \},$$
so that by Theorem \ref{AB1} one has
$$\MDeg_{R}(z)= z_1^2z_2+z_1^2z_3+z_1z_2z_3+z_1z_3^2+z_2z_3^2+z_3^3.$$
Each monomial in the multidegree corresponds to a minimal prime of the generic initial ideal, e.g $(2,1,0)$ corresponds to the ideal generated by the first  $2$ variables of the  first block and the first  variable of the second block. Hence  the multigraded generic initial ideal of $J(V)$ is the intersection of the $6$ components: 
$$\begin{array}{ccccccc}
(2,1,0) & \to & (x_1,x_2,  x_3)  & \quad  &
(2,0,1)& \to & (x_1,x_2,  x_4)\\
(1,1,1)& \to & (x_1  ,x_3,  x_4) & \quad  &
(1,0,2)& \to & (x_1,x_4,x_5)\\
(0,1,2) & \to & (x_3,x_4,x_5) &  &
(0,0,3) & \to & (x_4,x_5,x_6)
\end{array}
$$
 i.e $\gin(J(V))=(x_1x_4, x_2x_4, x_3x_4, x_1x_5, x_2x_3x_5, x_1x_3x_6)$. 
\end{ex}

\section{Multiview ideals} 
\label{S4}

In this section, we turn our attention to multiview ideals. Consider a collection of matrices with scalar entries $A=\{A_i\}_{i=1,\dots,m}$ with $A_i$ of size $d_i\times n$ and $\rank A_i=d_i$.  One has an induced  rational map $$\phi_A: \PP^{n-1} \dashrightarrow  \prod \PP^{d_i-1}$$ sending $x\in \PP^{n-1}$ to $(A_ix)_{i=1,\dots, m}$. 
The ideal $J_A$ associated to the closure of the image of $\phi_A$ is called multiview ideal. We refer to \cite{AST} for a discussion of the role played by $J_A$ in various aspects of geometrical computer vision. Our goal is proving the 

\begin{thm}\label{bingthm}
For all   choices of  $A=\{A_i\}_{i=1,\dots,m}$ the multiview ideal  $J_A$ is a CS ideal and it defines a Cohen-Macaulay normal domain. 
\end{thm} 

Theorem \ref{bingthm} is proved in \cite{AST} in the case $n=4$ and $d_i=3$ for all $i$, under the assumption that the $A_i$'s are generic. Later on Binglin Li in his yet unpublished preprint \cite{Binglin} proved Theorem~\ref{bingthm} in general. Furthermore he gave a combinatorial description of the multidegree and the generators of $J_A$. Our goal is giving two alternative proofs of Theorem~\ref{bingthm}. 

First of all we introduce  the algebraic objects needed to describe the problem. Notice that our point of view is somewhat dual to that of~\cite{Binglin}.  
 We denote by  $V_i$ the vector space of linear forms of $T=K[x_1,\dots, x_n]$ generated by the entries of the matrix $A_ix$ where $x$ is the column vector with entries $x_1,\dots, x_n$.  
Then   $V_1,\dots, V_m$ is  a collection of vectors spaces of linear forms of dimension $d_i=\dim_K V_i$. We define 
$$A(V_1,\dots,V_m)=K[V_1y_1,\dots, V_my_m] \subset T[y_1,\dots,y_m],$$ 
i.e. $A(V_1,\dots,V_m)$ is the subalgebra of the polynomial ring $T[y_1,\dots,y_m]$ generated by the elements $vy_i$ with $v\in V_i$. By construction   $A(V_1,\dots,V_m)$ is the multigraded coordinate ring of the closure of the image of $\phi_A$. The  $\ZZ^m$-graded structure on   $A(V_1,\dots,V_m)$ is induced by the assignment  $\deg y_i=e_i\in \ZZ^m$. 

We present  $A(V_1,\dots,V_m)$ as a quotient of $K[x_{ij} : i=1,\dots,m, \ \  j=1,\dots, d_i ]$ via the 
$K$-algebra surjection  
$$\phi: K[x_{ij} : i=1,\dots,m, \ \  j=1,\dots, d_i ]\to A(V_1,\dots,V_m)$$
defined by $\phi(x_{ij})=v_{ij}y_i$ where $\{v_{ij} : j=1,\dots, d_i\}$ is a basis of $V_i$.   
By construction $J_A=\ker \phi$. 

\begin{proof}[Proof of Theorem~\ref{bingthm}] {\it First proof:} We take the point of view of \cite{C}. Observe that $A(V_1,\dots,V_m)$ is a subring of the Segre product $K[x_iy_j : i=1,\dots,n \mbox{ and } j=1,\dots,  m]$ of the polynomial rings $T$ and $K[y_1,\dots, y_m]$.  The latter is defined as a quotient of the polynomial ring $K[x_{ij} :  i=1,\dots,m, \ \  j=1,\dots, m]$ by the ideal $I_2(X)$ of $2$-minors of the matrix $X=(x_{ij})$. 
Hence $J_A$ is obtained from $I_2(X)$ by performing a $\ZZ^m$-graded change of variables and then eliminating the variables $x_{ij}$ with $d_i<j\leq m$. Since $I_2(X)$ is a CS ideal, by~\cite[Theorem 1.6]{CDG2} we may conclude that $J_A$ is a CS ideal. 

\smallskip

\noindent{\it Second proof:} We consider the $K$-algebra map
$$\phi_0: K[x_{ij} : i=1,\dots,m, \ \  j=1,\dots, d_i ]\to T$$
defined by $\phi_0(x_{ij})=v_{ij}$.   Clearly $\Ker\phi_0$ is generated by linear forms, indeed by $\sum_{i=1}^m d_i-\dim_K \sum_{i=1}^m  V_i$ linear forms. 
By construction $\Ker \phi$ is the ideal generated by the $\ZZ^m$-homogeneous elements of $\Ker\phi_0$. With the notation introduced above: 
$$\Ker \phi=(\Ker\phi_0)^\star$$
and by Theorem~\ref{AB1} we conclude that $\Ker \phi$ is CS, i.e. $J_A$ is CS. 

Finally, Cohen-Macaulayness and normality follow from Brion's Theorem \ref{Brion}. 
\end{proof} 

\begin{ex} Let $m\leq d$ and $n=(m-1)d$. Let $A=\{A_i\}_{i=1,\dots,m}$ with $A_i$ generic of size $d\times n$.   
By genericity, we may choose coordinates such that  
$V_j=\langle x_{d(j-1)+1},x_{d(j-1)+2} ,\dots, x_{jd}\rangle $ for $j=1,\dots, m-1$ and 
$V_m=\langle v_1,\dots,  v_d\rangle$, with $v_h=-\sum_{j=1}^{m-1} x_{d(j-1)+h}$  for $h=1,\dots,d$.  
Then $\Ker \phi_0=(  \sum_{i=1}^m x_{ik} : k=1,\dots,d )$. It follows that the multiview ideal $J_A$, i.e. the ideal of the closure of the image of the rational map $\phi_A: \PP^{n-1} \dashrightarrow  \prod_{i=1}^m \PP^{d-1}$, 
is defined by the $m$-minors of the generic $m\times d$ matrix 
$$\left (
\begin{array}{ccccc}
x_{11} & x_{12} & \dots & x_{1d} \\
x_{21} & x_{22} & \dots & x_{2d} \\
\dots& \dots& \dots& \dots  \\
x_{m1} & x_{m2} & \dots & x_{md} \\
 \end{array}
\right ).
$$
\end{ex}

\end{document}